\documentclass[a4paper]{article}

\usepackage{amssymb}
\usepackage{amsmath}
\usepackage{amsfonts}
\usepackage{amsthm}
\usepackage{mathabx}
\usepackage{graphicx}
\usepackage{cite}
\usepackage[breaklinks=true]{hyperref}
\usepackage{graphicx}
\usepackage{tikz}
\usepackage{placeins}
\usepackage{float}
\usepackage{pstricks}
\usepackage{tikz-cd}

\usetikzlibrary{arrows,chains,matrix,positioning,scopes}

\newtheorem{thm}{Theorem}[section]
\newtheorem{theorem}[thm]{Theorem}
\newtheorem{lem}[thm]{Lemma}

\newtheorem{prop}[thm]{Proposition}
\newtheorem{cor}[thm]{Corollary}

\newenvironment{proofofthm}[1]{\noindent{\it{Proof of Theorem \ref{#1}.}}}{$\qed$}

\theoremstyle{definition}
\newtheorem{definition}[thm]{Definition}

\newtheorem{notation}[thm]{Notation}

\newtheorem{innercustomthm}{Theorem}
\newenvironment{customthm}[1]
  {\renewcommand\theinnercustomthm{#1}\innercustomthm}
  {\endinnercustomthm}

\theoremstyle{remark}
\newtheorem{rem}{Remark}

\newtheorem*{prop proof}{Proof of Proposition}

\newcommand{\RR}{\mathbb{R}}      
\newcommand{\ZZ}{\mathbb{Z}}        
\newcommand{\EE}{\mathbb{E}}      
\newcommand{\lk}{{\rm{Lk}}}

\numberwithin{equation}{section}

\title{Almost Hyperbolic Groups with Almost Finitely Presented Subgroups}
\author{Robert Kropholler}

\begin{document}
\maketitle
\begin{abstract}
	We construct new examples of CAT(0) groups containing non finitely presented subgroups that are of type $FP_2$, these CAT(0) groups do not contain copies of $\ZZ^3$. We also give a construction of groups which are of type $F_n$ but not $F_{n+1}$ with no free abelian subgroups of rank greater than $\lceil\frac{n}{3}\rceil$. 
\end{abstract}

\section{Introduction}

Subgroups of CAT(0) groups can be shown to exhibit interesting finiteness properties. In \cite{bestvina_morse_1997} there are examples of groups satisfying $F_n$ but not $F_{n+1}$ for all $n$ as well as the first examples of groups of type $FP_2$ that are not finitely presented answering a question of Brown. 

\begin{definition}
A group $G$ satisfies {\em property $F_n$}, if there is a classifying space $K(G,1)$ with finite $n$ skeleton.
\end{definition}

\begin{definition}
A group $G$ is of {\em type $FP_n$} if there is a partial resolution of the trivial $\ZZ G$ module, 

\begin{figure*}[h]
\center
\begin{tikzpicture}[node distance=1.2cm, auto]
\node (Z) {$\ZZ$};
\node (P_0) [left of=Z] {$P_0$};
\node (P_1) [left of=P_0] {$P_1$};
\node (dots) [left of=P_1] {$\dots$};
\node (P_n) [left of=dots] {$P_n$};
\node (0r) [right of=Z] {$0$}; 

\draw[->,] (P_1) to node {} (P_0);
\draw[->,] (P_0) to node {} (Z);
\draw[->,] (dots) to node {} (P_1);
\draw[->,] (P_n) to node {} (dots);
\draw[<-,] (0r) to node {} (Z);
\end{tikzpicture}
\end{figure*}
where $P_i$ is a finitely generated projective $\ZZ G$ module.
\end{definition}

It can easily be seen that $F_1$ is equivalent to finite generation and $F_2$ is equivalent to finite presentability. In \cite{bestvina_morse_1997} the groups are kernels of maps $\phi:G\to \ZZ$ where $G$ is a right angled Artin group, the finiteness properties of the kernel depend solely on the defining flag complex for the right angled Artin group. 

The construction of groups of type $FP_2$ that are not finitely presented are contained in right angled Artin groups having free abelian subgroups of rank 3. Since \cite{bestvina_morse_1997} there have been other constructions of groups with subgroups of type $FP_2$ not $F_2$ see \cite{bux_bestvina-brady_1999other, leary_uncountably_2015}. In many cases these groups are subgroups of groups $G$ of type $F$, in these cases the maximal rank of a free abelian subgroup of $G$ is at least 3. We construct the first examples of groups of type $F$ containing no free abelian subgroups of rank 3 containing a subgroup of type $FP_2$ that is not $F_2$. 

\begin{customthm}{A}
There exists a non positively curved space $X$ such that $\pi_1(X)$ contains no subgroups isomorphic to $\ZZ^3$, which contains a subgroup of type $FP_2$ not $F_2$. 
\end{customthm}

The groups constructed in \cite{bestvina_morse_1997} which are of type $F_{n-1}$ not $F_n$ all contain a free abelian subgroup of rank $n-1$. Brady asked in \cite[Question 8.5]{bestvina_problem} whether there exist groups of type $F_{n-1}$ but not $F_{n}$ which do not contain $\ZZ^2$; he notes that the known examples all contain $\ZZ^{n-2}$. While we are not able to find examples without $\ZZ^2$, we reduce the bound to a fraction of $n$.  

\begin{customthm}{B}
For every positive integer $n$, there exists a group of type $F_{n-1}$ but not $F_{n}$ that contains no abelian subgroups of rank greater than $\lceil \frac{n}{3}\rceil$.
\end{customthm}

In the future, we would like to extend these theorems, giving examples of hyperbolic groups with such subgroups.

The author would like to thank Federico Vigolo for kindly helping draw several of the figures contained within. The author would also like to thank Martin Bridson and Gareth Wilkes for reading earlier drafts of this paper and providing helpful and constructive comments.

\section{Preliminaries}
\subsection{Cube Complexes}

A cube complex can be constructed by taking a collection of disjoint cubes and gluing them together by isometries of their faces.

There is a standard way in which cube complexes can be endowed with metrics. There is a well known criterion of \cite{chern_hyperbolic_1987} that characterises those cube complexes that are locally CAT(0).

The precise definition of a cube complex is given in \cite[Def. 7.32]{bridson_metric_1999} as follows.

\begin{definition}
A {\em cube complex} $X$ is a quotient of a disjoint union of cubes $K = \bigsqcup_{C}[0,1]^{n_c}$ by an equivalence relation $\sim$. The restrictions $\chi_c:[0,1]^{n_c}\to X$ of the natural projection $\chi:K\to X = K/\sim$ are required to satisfy:
\begin{itemize}
\item for every $c\in C$ the map $\chi_c$ is injective;
\item if $\chi_c([0,1]^{n_c})\cap \chi_{c'}([0,1]^{n_{c'}})\neq \emptyset$, then there is an isometry $h_{c,c'}$ from a face $T_c\subset [0,1]^{n_c}$ onto a face $T_{c'}\subset [0,1]^{n_c'}$ such that $\chi_c(x) = \chi_{c'}(x')$ if and only if $x' = h_{c,c'}(x)$.
\end{itemize}
\end{definition}

\begin{definition}
A metric space is {\em non-positively curved} if its metric is locally CAT(0)
\end{definition}

Gromov's insight allows us to easily check whether a $P\EE$ complex is non-positively curved. 

\begin{lem}[Gromov, \cite{chern_hyperbolic_1987}, p. 120]
A $P\EE$ complex is non-positively curved if and only if the link of each vertex is a CAT(1) space. 
\end{lem}

In a cube complex the link of each vertex is a spherical complex built from all-right speherical simplices (see \cite[\S5.18]{bridson_metric_1999}). Gromov realised that an all-right spherical complex is CAT(1) if and only if it is flag.

\begin{definition}
A complex $L$ is a {\em flag complex} if it is simplicial and every set $\{v_1, \dots, v_n\}$ of pairwise adjacent vertices spans a simplex. 
\end{definition}

Flag complexes are completely determined by their $1$-skeleta. 

Thus we arrive at the following combinatorial condition for cube complexes.

\begin{lem}[Gromov, \cite{chern_hyperbolic_1987}]
A cube complex is non-positively curved if the link of every vertex is a flag complex. 
\end{lem}

We wish to limit the rank of free abelian subgroups, in fact we will limit the largest dimension of an isometrically embedded flat. The following theorem shows that in the CAT(0) world the maximal dimension of a flat is the same as the maximal rank of a free abelian subgroup. 

\begin{thm}[Bridson-Haefliger, \cite{bridson_metric_1999}, II.7]\label{flattorus}
Let $X$ be a compact non positively curved cube complex, $\tilde{X}$ its universal cover and $G = \pi_1(X)$. If $H = \ZZ^n$ is a subgroup of $G$, then there is an isometrically embedded copy of $i:\RR^n\hookrightarrow\tilde{X}$. Moreover, the quotient $i(\RR^n)/H$ is an $n$-torus
\end{thm}

Finally we would like to know when a cube complex is hyperbolic. Bridson shows that the only obstruction is containing an isometrically embedded flat. 

\begin{thm}[Bridson, \cite{bridson_existence_1995}, Theorem A]\label{flatplane}
Let $X$ be a compact non-positively curved cube complex and $\tilde{X}$ be its universal cover. $\tilde{X}$ is not hyperbolic if and only if there exists an isometric embedding $i:\EE^2\hookrightarrow \tilde{X}$. 
\end{thm}

\subsection{Right angled Artin groups}

Right angled Artin groups have been at the centre of a lot of recent study, particularly because of their interesting subgroup structure \cite{wise_structure_2011, agol_virtual_2013, haglund_special_2008, bestvina_morse_1997}. In particular, their subgroups have interesting connections with finiteness properties of groups, as shown in \cite{bestvina_morse_1997}. For completeness, we recall the basic theory of right angled Artin groups.

\begin{definition}
Given a flag complex $\Gamma$, we define the associated {\em right angled Artin group (RAAG)} $A_{\Gamma}$, as the group,
$$A_{\Gamma} = \bigl\langle \Gamma^{(0)}\bigm\vert[v,w] = 1 \mbox{ if }[v,w]\in\Gamma^{(1)}\bigr\rangle.$$
\end{definition}

These groups have non-positively curved cube complexes as classifying spaces, which are unions of tori:

\begin{definition}
Given a flag complex $\Gamma$ the {\em Salvetti complex} $S_{\Gamma}$ is defined as follows. 

For each vertex $v_i$ in $\Gamma$, let $S^1_{v_i} = S^1$ be a copy of the circle cubulated with 1 vertex. For each simplex $\sigma = [v_0, \dots, v_n]$ of $\Gamma$ there is an associated torus $T_{\sigma} = S^1_{v_0}\times\dots\times S^1_{v_n}$. If $\tau <\sigma$, then there is a natural inclusion $T_{\tau}\hookrightarrow T_{\sigma}$. Now define

$$S_{\Gamma} = \left.\coprod_{\sigma<\Gamma} T_{\sigma}\middle/\sim\right.$$
where the equivalence relation $\sim$ is generated by the inclusions $T_{\tau}\hookrightarrow T_{\sigma}$. 
\end{definition}

There is a map $S_{\Gamma}\to (S^1)^{\Gamma^{(0)}}$ sending each circle in $S_{\Gamma}$ to the respective circle in $(S^1)^{\Gamma^{(0)}}$ and extending linearly over cubes; this map is an inclusion of cubical complexes.

It is a standard fact that $S_{\Gamma}$ has fundamental group $A_{\Gamma}$ and is a non-positively curved cube complex. Proofs of the following can be found in \cite{bestvina_morse_1997}.

\begin{lem}
$\pi_1(S_{\Gamma}) = A_{\Gamma}$. 
\end{lem}

\begin{lem}
$S_{\Gamma}$ is non-positively curved. 
\end{lem}

\subsubsection{A new classifying space}\label{construct}

We will construct a new classifying space, which will also be a non-positively curved cube complex and will be more amenable to taking branched covers. We will build such a classifying space in the case that the flag complex satisfies the following condition. 

\begin{definition}
A simplicial complex $\Gamma$ has an {\em $n$-partite structure} if $\Gamma$ is contained in the join $V_1\ast\dots\ast V_n$ of several discrete sets $V_i$. In the case $n=2$ we will say {\em bipartite} and in the case $n=3$, {\em tripartite}. 
\end{definition}

Any finite simplicial complex can be given an $n$-partite structure as a subcomplex of a simplex. 

We define a cube complex $K_{\Gamma}$ for an $n$-partite complex $\Gamma$ as follows.

Let $V_i = \{v_1^i, \dots, v_{m_i}^i\}$ and $I = \{1, \dots, n\}$. For each vertex $v_k^i$, let $S^1_{v_k^i}$ be a copy of $S^1$ cubulated with two vertices labelled 0 and 1, and two edges $e_{v_k^i}$ and $e_{v_0^i}$. 

\begin{definition}\label{classpace}
Given $J\in \mathcal{P}(I)$, $J = \{i_1, \dots, i_l\}$, we say that $\sigma$ is a {\em $J$-simplex} if it is of the form $[v_{k_1}^{i_1}, \dots, v_{k_l}^{i_l}]$.
\end{definition}

\begin{rem}
Every simplex is a $J$ simplex for some unique (possibly empty) $J$.
\end{rem}

For each $J$-simplex $\sigma = [v_{k_1}^{i_1}, \dots, v_{k_l}^{i_l}]$ in $\Gamma$ we associate the following space. 
$$T_{\sigma} = \left(\prod_{v\in\sigma}S^1_{v}\right)\times\left(\prod_{i\notin J}e_{v_0^i}\right).$$
This is the product of a torus and a cube. 

Let $\mathcal{T}_J = \{T_{\sigma}:\sigma$ is a $J$-simplex$\}$ and $\mathcal{T} = \coprod_{J\in\mathcal{P}(I)}\mathcal{T}_J$. Given simplices $\tau<\sigma$ there is a natural inclusion $T_{\tau}\hookrightarrow T_{\sigma}$. 
$$K_{\Gamma} = \left.\coprod_{T\in\mathcal{T}}T\middle/\sim\right..$$

Where once again the equivalence relation is generated via the inclusions $T_{\tau}\to T_{\sigma}$. We need to prove the two key lemmata: the fundamental group of $K_{\Gamma}$ is $A_{\Gamma}$, and $K_{\Gamma}$ is non-positively curved. 

\begin{lem}
$\pi_1(K_{\Gamma}) = A_{\Gamma}$.
\end{lem}
\begin{proof}
Apply the Seifert-van Kampen Theorem repeatedly.
\end{proof}

\begin{lem}
$K_{\Gamma}$ is non-positively curved. 
\end{lem}
\begin{proof}
There are $2^n$ vertices in $K_{\Gamma}$. Given two vertices $v, w$ there is a cellular isomorphism $K_{\Gamma}\to K_{\Gamma}$ which sends $v$ to $w$. Thus we only need to check the link of one vertex; we will check the link of $\mathbf{0} = (0,\dots, 0)$. Let $L = {\rm Lk}(\mathbf{0}, K_{\Gamma})$. Let $\overline{V_i} = V_i\cup \{v_0^i\}$. 

There is a vertex in $L$ for each edge at $\mathbf{0}$, so $L^{(0)} = \coprod \overline{V_i}$. Two distinct vertices $v_{k_1}^{i_1}$ and $v_{k_2}^{i_2}$ are connected if one of the following conditions holds:
\begin{enumerate}
\item $k_1 = 0, k_2 = 0$,
\item $k_1 = 0, i_1\neq i_2$,
\item $i_1\neq i_2, k_2 = 0$,
\item $[v_{k_1}^{i_1}, v_{k_2}^{i_2}]$ is an edge of $\Gamma$.
\end{enumerate}
These edges come from the following subcomplexes. 
\begin{enumerate}
\item $T_{\emptyset}$,
\item $T_{v_{k_1}^{i_1}}$,
\item $T_{v_{k_2}^{i_2}}$,
\item $T_{[v_{k_1}^{i_1}, v_{k_2}^{i_2}]}$.
\end{enumerate}

This tells us that $L\subset \overline{V_1}\ast\dots\ast\overline{V_n}$. We now want to prove that $L$ is in fact a flag complex. Given a set $W = \{v_{k_1}^{i_1}, \dots v_{k_l}^{i_l}\}$ of pairwise adjacent vertices in $L$ we want to show that $[v_{k_1}^{i_1}, \dots v_{k_l}^{i_l}]$ is also in $L$. We split $W$ into two sets $W_0 = \big\{v_{k_j}^{i_j}: k_j = 0\big\}$ and $W_1 = W\smallsetminus W_0 = \{w_1, \dots, w_a\}$. Since the vertices in $W_1$ are pairwise adjacent there is a simplex $\sigma\subset\Gamma$ spanning them. We see that the subcomplex $T_{\sigma}$ contains an $(l+1)$-cell which fills the required simplex in $L$. 
\end{proof}

It should be noted the natural projections are injective on each closed cube and not just on the interior. 

\begin{rem}
The complex $K_{\Gamma}$ requires a choice of $n$-partite structure. Given two $n$-partite structures the complexes are homotopy equivalent but are not isomorphic as cube complexes. This is shown in Figure \ref{spaceex}, together with some examples of the construction.
\end{rem}

\begin{figure}\center
\def\svgwidth{150mm}

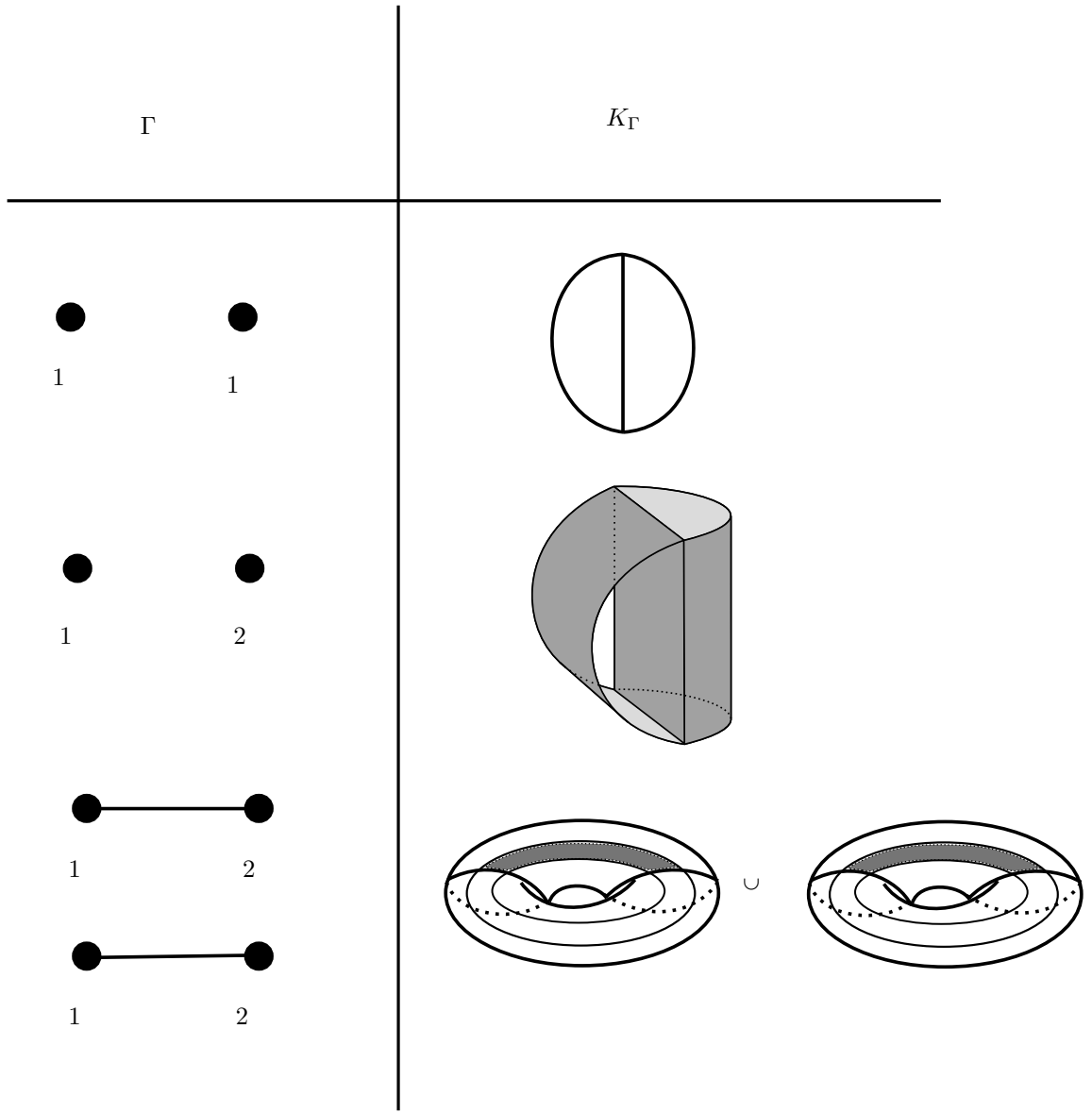
\caption{Some examples of $K_{\Gamma}$. The labels on vertices exhibit the $n$-partite structure. In the last example the shaded regions are identified.}
\label{spaceex}
\end{figure}

\begin{rem}\label{keyrem}
Let $\Theta_i$ be a ``cage graph'' with two vertices that has an edge for each element of $\overline{V_i}$, then the complex $K_{\Gamma}$ constructed above embeds in $\prod_{i=1}^n\Theta_i$.
\end{rem}

 This fact will come in useful later, as will the fact that in the case $n = 3$ there is an embedded copy of $\Theta_1\amalg\Theta_2\amalg\Theta_3$ in $K_{\Gamma}$.

\subsubsection{Finiteness properties of subgroups of RAAGs}

We will require one theorem on the finiteness properties of subgroups of RAAGs from \cite{bux_bestvina-brady_1999other}. A homomorphism $f:A_{\Gamma}\to\ZZ$ can be defined by putting an integer label on each vertex and sending the corresponding generator of $A_{\Gamma}$ to its label. 
\begin{definition}
Let $f\colon A_{\Gamma}\to \ZZ$ be a homomorphism. We denote $\Gamma^{\dagger}<\Gamma$ the full subcomplex spanned by those vertices with label $0$. Let $L^{\ast}$ be the full subcomplex spanned by vertices not in $\Gamma^{\dagger}$. 
\end{definition}

\begin{thm}[Bux-Gonzalez \cite{bux_bestvina-brady_1999other}, Theorem A]\label{buxthm}
Let $f\colon A_{\Gamma}\to\ZZ$ be a homomorphism. Then the following are equivalent:
\begin{itemize}
\item The kernel of $f$ is of type $FP_n$, respectively $F_2$. 
\item For every, possibly empty, dead simplex $\sigma < \Gamma^{\dagger}$ the complex $\Gamma^{\ast}\cap \lk(\sigma, \Gamma)$ is homologically $(n-{\rm dim}(\sigma)-2)$-connected, respectively $L^{\ast}$ is simply connected. 
\end{itemize}
\end{thm}

\subsection{Branched Covers of Cube Complexes}

We will take branched covers of cube complexes to get rid of high dimensional flats. The techniques we will use were developed in \cite{brady_branched_1999}. The idea is to take an appropriate subset which intersects the high dimensional flats.

\begin{definition}
Let $X$ be a non-positively curved cube complex. We say that $Y\subset X$ is a {\em branching locus} if it satisfies the following conditions:
\begin{enumerate}
\item $Y$ is a locally convex cubical subcomplex, 
\item $Lk(c, X)\smallsetminus Y$ is connected and non-empty for all cubes $c$ in $Y$.
\end{enumerate}
\end{definition}

The first condition is required to prove that non-positive curvature is preserved when taking branched covers. The second is a reformulation of the the classical requirement that the branching locus has codimension 2 in the theory of branched covers of manifolds, ensuring that the trivial branched covering of $X$ is $X$. 

\begin{definition}
A {\em branched cover} $\hat{X}$ of $X$ over the branching locus $Y$ is the result of the following process.
\begin{enumerate}
\item Take a finite covering $\overline{X\smallsetminus Y}$ of $X\smallsetminus Y$.
\item Lift the piecewise Euclidean metric locally and consider the induced path metric.
\item Take the metric completion $\hat{X}$ of $\overline{X\smallsetminus Y}$.
\end{enumerate}
\end{definition}

We require some key results from \cite{brady_branched_1999} which allow us to conclude that this process is natural and that the resulting complex is still a non-positively curved cube complex. 

\begin{lem}[Brady, \cite{brady_branched_1999}, Lemma 5.3]
There is a natural surjection $b: \hat{X}\to X$ and $\hat{X}$ is a piecewise Euclidean cube complex. 
\end{lem}

\begin{lem}[Brady, \cite{brady_branched_1999}, Lemma 5.5]\label{npcbranch}
If $Y$ is a finite graph, then $\hat{X}$ is non-positively curved. 
\end{lem}

\section{Hyperbolisation in dimension 2} \label{dim2hyp}

This section will be a warm up to our key theorems which are all related to dimension 3. The lower dimensional case carries a lot of the ideas that will be used later.

Throughout this section $\Gamma$ will be a bipartite graph; say $\Gamma\subset V_1\ast V_2$. Let $V_i = \{v^i_1, \dots, v^i_{m_i}\}$. Let $\Theta_i$ be the graph with two vertices,  labelled $0$ and $1$, and $|V_i|+1$ edges, labelled $v_0^i, v_1^i,\dots, v_{m_i}^i$, all directed from $0$ to $1$. We noted in Remark \ref{keyrem} that the complex $K_{\Gamma}$ constructed in Section \ref{construct} is a subcomplex of $\Theta_1\times\Theta_2$.

\begin{thm}\label{raaghyp2}
Let $\Gamma$ be a bipartite graph, let $A_{\Gamma}$ be the associated RAAG and let $K_{\Gamma}$ be the classifying space constructed in Section \ref{construct}. Then there is a branched cover $R_{\Gamma}$ of $K_{\Gamma}$ which has hyperbolic fundamental group.
\end{thm}
There are in fact many branched covers delivering this result. During the course of the proof we will pick a prime $p$ and as this varies different hyperbolic branched covers are obtained.
\begin{proof}
The branching locus will be one of the 4 vertices of $K_{\Gamma}$. Given two vertices $v$ and $w$ there is a homeomorphism of $K_{\Gamma}$ which sends $v$ to $w$. Therefore, it does not matter which vertex we pick to be our branching locus. We will choose the vertex $\mathbf{0} = (0,0)$. 

\begin{figure}\center
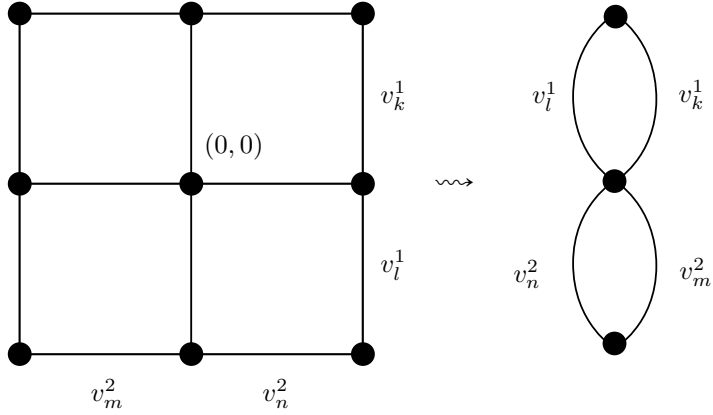
\caption{Depiction of the deformation retraction $T^2\smallsetminus\{(0,0)\}\to S^1\vee S^1$.}
\label{torus}
\end{figure}

$K_{\Gamma}\smallsetminus\{\mathbf{0}\}$ deformation retracts onto the graph $\Theta_1\vee\Theta_2$. This can be seen as follows. We start with a torus cubulated as in Figure \ref{torus}. From each torus we remove the center vertex. The complement deformation retracts onto the graph depicted in Figure \ref{torus}. We now identify the edges via their labels, which will result in $\Theta_1\vee\Theta_2$. 

This argument shows that $\pi_1(K_{\Gamma}\smallsetminus\{\mathbf{0}\})$ is a free group on $|V_1|+|V_2|$ generators. We will denote these generators $a_i = v_{i-1}^1\overline{v_i^1}$ for $i\in\{1, \dots ,|V_1|\}$ and $b_j = v_{j-1}^2\overline{v_j^2}$ for $j \in\{ 1,\dots, |V_2|\}$.

From the deformation described we get a map $\lk(\mathbf{0}, K_{\Gamma})\to \Theta_1\vee\Theta_2$. Each loop of length 4 in $\lk(\mathbf{0}, K_{\Gamma})$ corresponds to a torus as in Figure \ref{torus}. Under the deformation retraction this gets sent to a loop in $\Theta_1\vee\Theta_2$ corresponding to the commutator $[v_k^1\overline{v_l^1}, v_m^2\overline{v_n^2}] = [\prod_{i=k+1}^{k+l}a_i, \prod_{j=m+1}^{m+n}b_j]$.

Let $p>1+\max(|V_1|, |V_2|)$ be a prime and $S_p$ be the symmetric group on $p$ letters. Let $\lambda$ be a $p$-cycle in $S_p$ and $\mu$ an element which conjugates $\lambda$ to $\lambda^r$, where $r$ is a generator of $\ZZ_p^{\times}$. We define our cover using the map, 
\begin{align*}
\rho:\pi_1(\Theta_1\vee\Theta_2)&\to S_p,\\
a_i&\mapsto \lambda,\\
b_j&\mapsto\mu,
\end{align*}
taking the cover corresponding to the stabiliser of $1$ in $S_p$. 

Note that the commutator $[\lambda^m, \mu^n] = \lambda^{(r^n-1)m}$ is a $p$-cycle. This means that the loops of length 4 in the link have connected preimage in the cover. We take the completion $R_{\Gamma}$ of the resulting complex; there is a natural map $b:R_{\Gamma}\to K_{\Gamma}$. The link of the vertex which maps to $\mathbf{0}$ will contain no cycles of length $<5$.

We now prove that the resulting complex has hyperbolic universal cover. We know that $\widetilde{R}{_{\Gamma}}$ is a non-positively curved cube complex. So to prove that it is hyperbolic we just have to show that there are no isometric embeddings $\EE^2\to \tilde{R_{\Gamma}}$ by Theorem \ref{flatplane}. If there is such an embedding, then it will contain at least one square 2-cell. However, each square contains one vertex which is a lift of $\mathbf{0}$ and in the link of this vertex there are no loops of length $2\pi$. However, were the flat plane to be isometrically embedded there would be a loop of length $2\pi$ in the link of every vertex on the plane. This contradiction completes the proof.
\end{proof}

We use this theorem along with Morse theoretic ideas from Section \ref{morsetheory} to find more examples of hyperbolic groups with finitely generated subgroups that are not finitely presentable. 

\begin{prop}\label{hypf1}
Let $\Gamma$ be a complete bipartite graph on sets $A = \{a_1, a_2, a_3\}$ and $B = \{b_1, b_2, b_3\}$, fix $p$ satisfying the above hypotheses and let $b:R_{\Gamma}\to K_{\Gamma}$ be the branched cover constructed above. Then $\pi_1(R_{\Gamma})$ has a finitely generated subgroup which is not finitely presentable. 
\end{prop}
\begin{proof}
In this case $K_{\Gamma} = \Theta_1\times\Theta_2$ where $\Theta_i$ is as above and has 4 edges. Put an orientation on each edge of $\Theta_i$ such that there are 2 edges oriented towards each vertex and 2 edges oriented away from each vertex. 
Cubulating $S^1$ with one vertex and one oriented edge. Define maps $h_i:\Theta_i\to S^1$ on each edge by their orientation. Define $h:K_{\Gamma}\to S^1$ by $h(x_1, x_2) = h_1(x_1)+h_2(x_2)$. Precomposing with the branched covering map $b$ and lifting to universal covers, we obtain a Morse function $f:\tilde{R}_{\Gamma}\to \RR$, which is $(hb)_*$-equivariant. 

The ascending and descending links of this Morse functions are the preimages, under $b$, of the ascending and descending links of $h$. For the Morse function $h$ the ascending and descending links are joins of the ascending and descending links for $h_i$. These will be copies of $S^0$, so the ascending and descending links will be copies of $S^1$. 

There are now two possibilities, if we look at a vertex in $R_{\Gamma}$ which does not map to $\mathbf{0}$ the ascending and descending links will remain unchanged and will still be copies of $S^1$. 

If the vertex in question maps to $\mathbf{0}$ we will study the ascending link the other case being identical. The ascending link of $\mathbf{0}$ is a loop of length $4$. Taking a branched cover cause this loop of length 4 to lengthen but in the preimage it will still be a copy of $S^1$. It follows that the kernel of $(hb)_*$ is finitely generated but not finitely presentable by \ref{notfp}. 
\end{proof}

Sizeable graphs are used in \cite{kropholler_hyperbolic_2015} to give examples of hyperbolic groups with subgroups which are type $F_2$ not $F_3$. Here we outline a procedure for producing examples of sizeable graphs.

\begin{prop}\label{constructsize}
Let $A$ and $B$ be sets with partitions $A = A^+\sqcup A^-$ and $B = B^+\sqcup B^-$ where $A^-, B^-$ are non-empty and $|A^+|, |B^+|>1$. Let $\Gamma $ be the complete bipartite graph on $A$ and $B$. Let $A_{\Gamma}$ be the associated right angled Artin group, $K_{\Gamma}$ the classifying space from \ref{classpace} and $R_{\Gamma}$ the branched covering constructed in Theorem \ref{raaghyp2}. Let $v\in R_{\Gamma}$ be a vertex mapping to $(0,0)$. Then $Lk(v, R_{\Gamma})$ is sizeable. 
\end{prop}
\begin{proof}
The link of a vertex is a cover $\overline{\Lambda}$ of the graph $\Lambda$. $\Lambda$ is the complete bipartite graph on sets $A' = A\cup\{a_0\}$ and $B' = B\cup\{b_0\}$. Define $A'^+ = A^+$ and $A'^- = A'\smallsetminus A'^+$, defining $B'^+$ and $B'^-$ similarly.  

The graph $\overline{\Lambda}$ has the following properties. It is bipartite as it is the cover of a bipartite graph and it has no cycles of length 4 since the branching process was designed to remove these. To check the last property we let $A_{\overline{\Lambda}}^+$ be the set of vertices mapping to $A'^+\subset\Lambda$ and $A_{\overline{\Lambda}}^-$ the complement of these in the bipartite structure. We define $B_{\overline{\Lambda}}^+$ and $B_{\overline{\Lambda}}^-$ similarly. 

We must prove that $\overline{\Lambda}(A_{\overline{\Lambda}}^s\cup B_{\overline{\Lambda}}^t)$ is connected. $\Lambda(A'^s, B'^t)$ can be covered by finitely many loops $c_1, c_2, \dots$ of length 4 such that $c_{m+1}\cap \cup_{i=1}^m c_m \neq\emptyset$ for all $m$. When taking the branched cover each $c_i$ has connected preimage and the intersection will still be non empty so the resulting union will be connected.
\end{proof}

\section{Almost hyperbolisation in dimension 3}\label{almosthyp}

\begin{notation}
Given a tripartite complex $L\subset V_1\ast V_2\ast V_3$, let $L_{ij}$ be the full subcomplex spanned by the vertices of $V_i\cup V_j$. 
\end{notation}

The main theorem of this section is the following.

\begin{thm}
Let $\Gamma$ be a tripartite flag complex, $A_{\Gamma}$ the associated RAAG and $K_{\Gamma}$ the classifying space constructed in Section \ref{construct}. Then there exists a branched cover $X$ of $K_{\Gamma}$, such that $\pi_1{(X)}$ contains no subgroups isomorphic to $\ZZ^3$. 
\end{thm}
\begin{proof}
Our branching locus will be $$B = (\Theta_1\times\{0\}\times\{1\})\cup(\{1\}\times\Theta_2\times\{0\})\cup(\{0\}\times\{1\}\times\Theta_3),$$ where $\Theta_i$ is the graph with 2 vertices, $0$ and $1$, and $|V_i|+1$ edges, each directed from $0$ to $1$. We have three maps
\begin{align*}
p_1:K_{\Gamma}\smallsetminus B&\to K_{\Gamma_{12}}\smallsetminus\{(0,1)\},\\
p_2:K_{\Gamma}\smallsetminus B&\to K_{\Gamma_{23}}\smallsetminus\{(0,1)\},\\
p_3:K_{\Gamma}\smallsetminus B&\to K_{\Gamma_{31}}\smallsetminus\{(0,1)\},
\end{align*}
which are the restrictions of the projections $\Theta_1\times\Theta_2\times\Theta_3\to \Theta_i\times\Theta_j$. The maps from Section \ref{dim2hyp} give us three maps 
\begin{align*}
\pi_1(K_{\Gamma_{12}}\smallsetminus\{(0,1)\})&\to S_{q_{12}}\\
\pi_1(K_{\Gamma_{23}}\smallsetminus\{(0,1)\})&\to S_{q_{23}}\\
\pi_1(K_{\Gamma_{31}}\smallsetminus\{(0,1)\})&\to S_{q_{31}}.
\end{align*}
Here, $q_{ij}$ are the primes picked in the process of taking a branched cover in Theorem \ref{raaghyp2}. Let $q = q_{12}q_{23}q_{31}$. We can combine these permutation representations with the projection maps above to get a map $\pi_1(K_{\Gamma}\smallsetminus B)\to S_{q}$, which defines a $q$-fold cover of $K_{\Gamma}\smallsetminus B$ by taking the subgroup corresponding to the stabiliser of $1$ in $S_{q}$. We complete this cover to get our branched cover $X$. 

Let $\{i,j,k\} = \{1,2,3\}$. The maps $p_1, p_2, p_3$ are retractions to see this consider the natural map
$$K_{\Gamma_{ij}}\hookrightarrow K_{\Gamma_{ij}}\times\left\{\frac{1}{2}\right\}\hookrightarrow K_{\Gamma_{ij}}\times e_{v_0^k}\hookrightarrow K_{\Gamma}.$$
It follows that $\pi_1(K_{\Gamma}\smallsetminus B)\to\pi_1(K_{\Gamma_{ij}}\smallsetminus\{(0,1)\})$ is surjective. 
We will now consider what the link of a vertex in the branched cover is. We will restrict our attention to a vertex mapping to $(0,1,0)$ all the other cases are similar. 

\begin{figure}\center
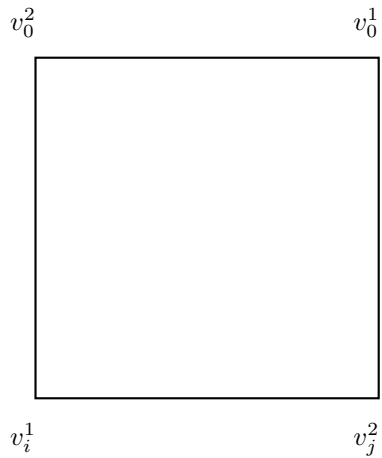
\caption{A loop corresponding to the commutator $[v_i^1, v_j^2]$.}
\label{loop1}
\end{figure}

We will consider the image of the link of $(0,1,0)$ under the three maps $p_1, p_2, p_3$. In the image of the map $p_1$, this link is sent surjectively onto the link of $(0,1)$ in $K_{\Gamma_{12}}$. By Section \ref{dim2hyp} we know that $K_{\Gamma_{12}}\smallsetminus\{(0,1)\}$ deformation retracts onto the graph $\Theta_1\vee\Theta_2$. Loops of length 4 are sent to commutators of the form $[v_i^1\overline{v}_j^1, v_k^2\overline{v}_l^2].$ In the map $\pi_1(K_{\Gamma_{12}}\smallsetminus\{(0,1)\})\to S_q$ this commutator is sent to a $q_{12}$-cycle. 

We must now consider the image under the maps $p_2, p_3$. These maps send the link to a disjoint union of contractible subsets, so the maps $\pi_1(K_{\Gamma}\smallsetminus B)\to S_{q_{23}}$ and $\pi_1(K_{\Gamma}\smallsetminus B)\to S_{q_{31}}$ send the image of the fundamental group of the link to the identity. 

From this we can see that, in the cover of $K_{\Gamma}\smallsetminus B$ corresponding to the stabiliser of $1$ in $S_{q}$, the preimage of one of the loops of length 4 depicted in Figure \ref{loop1} will have $q_{23}q_{31}$ components, and each component is a loop of length $4q_{12}$. We will now prove that there are no isometrically embedded planes of dimension $>2$. This, combined with Theorem \ref{flattorus}, will complete the proof of the theorem. 

\begin{figure}\center
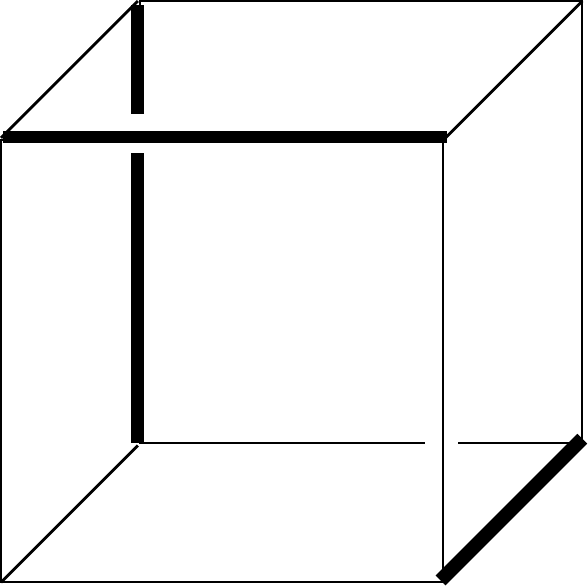
\caption{Intersection pattern of $\tilde{L}$ on a cube in $\tilde{X}$.}
\label{intpattern}
\end{figure}

Since the resulting cube complex has cubes of dimension at most 3, we can see that the dimension of an isometrically embedded flat plane is at most 3. If such a copy $E$ of $\EE^3$ were isometrically embedded in $\tilde{X}$, it would contain at least one cube and would in fact be a cubical embedding of the flat. For each vertex contained in the flat, the link would contain a subcomplex isomorphic to an octahedron. Let $\tilde{B}$ be the lift of the branching locus $B$. We can see how this intersects each cube in $\tilde{X}$ by Figure \ref{intpattern}. As such any 3-flat will intersect a vertex on $\tilde{B}$.

Let $x$ be a vertex on $\tilde{B}\cap E$. Then $\lk(x, \tilde{X})$ has a tripartite structure. If there is an octahedron in this complex it has a tripartite structure of the form $S^0\ast S^0\ast S^0$. One of the copies of $S^0$ will be contained in the vertices corresponding to $\tilde{B}$. The other 4 vertices form a loop of length 4 in the bipartite graph defined by the edges not in $\tilde{B}$. However, we constructed the branched cover so that this graph has no cycles of length $<6$. 
\end{proof}

\section{Bestvina-Brady Morse theory}\label{morsetheory}

While Bestvina-Brady Morse theory is defined in the more general setting of affine cell complexes, in this instance we shall only need it for non-positively curved cube complexes. 

For the remainder of this section, let $X$ be a CAT(0) cube complex and let $G$ be a group which acts freely, cellularly, properly and cocompactly on $X$. Let $\phi\colon G\to\ZZ$ be a homomorphism and let $\ZZ$ act on $\RR$ by translations.

Recall that $\chi_c$ is the characteristic map of the cube $c$.

\begin{definition}
We say that a function $f\colon X\to \RR$ is a {\em $\phi$-equivariant Morse function} if it satisfies the following 3 conditions.
\begin{itemize}
\item For every cube $c\subset X$ of dimension $n$, the map $f\chi_c\colon [0,1]^n\to\RR$ extends to an affine map $\RR^n\to\RR$ and $f\chi_c\colon [0,1]^n\to\RR$ is constant if and only if $n=0$.
\item The image of the $0$-skeleton of $X$ is discrete in $\RR$.
\item $f$ is $\phi$-equivariant, i.e. $f(g\cdot x) = \phi(g)\cdot f(x)$.
\end{itemize}
\end{definition}

We will consider the level sets of our function, which we will denote as follows. 

\begin{definition}
For a non-empty closed subset $I\subset\RR$ we denote by $X_I$ the preimage of $I$. We also use $X_t$ to denote the preimage of $t\in\RR$.
\end{definition}

The kernel $H$ of $\phi$ acts on the cube complex $X$ in a manner preserving each level set $X_{I}$. Moreover, it acts properly and cocompactly on the level sets. We will use the topological properties of the level sets to gain information about the finiteness properties of the group. We will need to examine how they vary as we pass to larger level sets. 

\begin{thm} [Bestvina-Brady, \cite{bestvina_morse_1997}, Lemma 2.3]
If $I\subset I'\subset\RR$ are closed intervals and $X_{I'}\smallsetminus X_{I}$ contains no vertices of $X$, then the inclusion $X_I\hookrightarrow X_{I'}$ is a homotopy equivalence. 
\end{thm}

If $X_{I'}\smallsetminus X_{I}$ contains vertices of $X$, then the topological properties of $X_{I'}$ can be very different from those of $X_I$. This difference is encoded in the ascending and descending links. 

\begin{definition}
The {\em ascending link} of a vertex is 

$Lk_{\uparrow}(v,X) = \bigcup \{{\rm Lk}(w,c)\mid\chi_c(w) = v$ and $w$ is a minimum of $f\chi_c\}\subset {\rm Lk}(v, X)$

The {\em descending link} of a vertex is 

${\rm Lk}_{\downarrow}(v,X) = \bigcup \{{\rm Lk}(w,c)\mid\chi_c(w) = v$ and $w$ is a {\small maximum} of $f\chi_c\}\subset {\rm Lk}(v, X)$
\end{definition}

\begin{thm} [Bestvina-Brady, \cite{bestvina_morse_1997}, Lemma 2.5]\label{homoeq}
Let $f$ be a Morse function. Suppose that $I\subset I'\subset\RR$ are connected, closed and $\min I = \min I'$ (resp. $\max I = \max I')$, and that $I'\smallsetminus I$ contains only one point $r$ of $f\big(X^{(0)}\big)$. Then $X_{I'}$ is homotopy equivalent to the space obtained from $X_I$ by coning off the descending (resp. ascending) links of $v$ for each $v\in f^{-1}(r)$. 
\end{thm}

We can now deduce a lot about the topology of the level and sub-level sets. We know how they change as we pass to larger intervals and so we have the following. 

\begin{cor}[Bestvina-Brady, \cite{bestvina_morse_1997}, Corollary 2.6]\label{cor1} Let $I,I'$ be as above.
\begin{enumerate}
\item If each ascending and descending link is homologically $(n-1)$-connected, then the inclusion $X_I\hookrightarrow X_{I'}$ induces an isomorphism on $H_i$ for $i\leq n-1$ and is surjective for $i=n$.
\item If the ascending and descending links are connected, then the inclusion $X_I\hookrightarrow X_{I'}$ induces a surjection on $\pi_1$. 
\item If the ascending and descending links are simply connected, then the inclusion $X_I\hookrightarrow X_{I'}$ induces an isomorphism on $\pi_1$.
\end{enumerate}
\end{cor}

Knowing that the direct limit of this system is a contractible space allows us to compute the finiteness properties of the kernel of $\phi$. 

\begin{theorem}[Bestvina-Brady, \cite{bestvina_morse_1997}, Theorem 4.1]\label{bbmorse}
Let $f\colon X\to \RR$ be a $\phi$-equivariant Morse function and let $H=\ker(\phi)$. If all ascending and descending links are simply connected, then $H$ is finitely presented (i.e. is of type $F_2$).
\end{theorem}

We would also like to have conditions which will allow us to deduce that $H$ does not satisfy certain other finiteness properties. A well known result in this direction is:

\begin{prop}[Brown, \cite{brown_cohomology_1982}, p. 193]\label{propb}
Let $H$ be a group acting freely, properly, cellularly and cocompactly on a cell complex $X$. Assume further that $\widetilde{H}_i(X,\ZZ)=0$ for $0\leq i\leq n-1$ and that $\widetilde{H}_n(X,\ZZ)$ is not finitely generated as a $\ZZ H$-module. Then $H$ is of type $FP_n$ but not $FP_{n+1}$.
\end{prop}

In \cite{brady_branched_1999}, the above result was used to prove that a certain group is not of type $FP_3$. However in our theorems not all the links will satisfy the assumptions of \cite[Theorem 4.7]{brady_branched_1999} and we require the following.

\begin{thm}[Kropholler, \cite{kropholler_hyperbolic_2015}, Theorem 2.23]\label{notfp}
Let $f\colon X\to \RR$ be a $\phi$-equivariant Morse function and let $H=\ker(\phi)$. Suppose that for all vertices $v$ the reduced homology of ${\rm Lk}_{\uparrow}(v)$ and ${\rm Lk}_{\downarrow}(v)$ is 0 in dimensions $0, \dots, n-1$ and $n+1$. Further assume that there is a vertex $v'$ such that $\widetilde{H}{_n}\big({\rm Lk}_{\uparrow}(v')\big)\neq 0$ or $\widetilde{H}{_n}\big({\rm Lk}_{\downarrow}(v')\big)\neq 0$ (possibly both). Then $H$ is of type $FP_n$ but not of type $FP_{n+1}$.
\end{thm}

Finally, we prove a key theorem regarding ascending and descending links. 

\begin{thm}\label{prodmorse}
For $i = 1, 2$ let $X_i$ be a CAT(0) cube complex, $G_i$ a group acting freely, properly and cocompactly on $X_i$, $\phi_i\colon G_i\to \ZZ$ a surjective homomorphism and $f_i\colon X_i\to \RR$ a $\phi_i$-equivariant Morse function. Then, there is a Morse function $f\colon X_1\times X_2\to \RR$ such that $${\rm Lk}_{\uparrow}\big((v_1, v_2), X_1\times X_2\big) = {\rm Lk}_{\uparrow}(v_1, X_1)\ast {\rm Lk}_{\uparrow}(v_2, X_2)$$ and $${\rm Lk}_{\downarrow}\big((v_1, v_2), X_1\times X_2\big) = {\rm Lk}_{\downarrow}(v_1, X_1)\ast {\rm Lk}_{\downarrow}(v_2, X_2)$$.
\end{thm}

We will prove the staement for ascending links, the proof for descending links is the same. The proof really relies on the following key lemma. 

\begin{lem}\label{asclink}
Let $X$ be a CAT(0) cube complex with Morse function $f$. Then ${\rm Lk}_{\uparrow}(v, X)$ is the full subcomplex of ${\rm Lk}(v, X)$ spanned by the vertices in ${\rm Lk}_{\uparrow}(v, X)$. 
\end{lem}

At first sight this does not appear to be an improvement, however, it allows for simple calculation of the ascending and descending links once the link of a vertex is known. 

\begin{proof}
Let $v_1, \dots, v_n$ be $n$ pairwise adjacent vertices in ${\rm Lk}_{\uparrow}(v, X)$, proving that the simplex $[v_1, \dots, v_n]$ is in ${\rm Lk}_{\uparrow}(v, X)$ will prove the claim. 

Let $e_i$ be the edge in $X$ corresponding to $v_i$. We must prove that $v$ is a minimum for $f$ restricted to $c = e_1\times\dots\times e_n$. We note that since $f$ extends to an affine map, $c$ is foliated by level sets $f^{-1}(t)$ for $t\in\RR$. Each of these level sets corresponds to a linear subspace of dimension $n-1$ not containing any subcubes of dimension $>0$. When intersected with the cube $c$ there are exactly two subspaces where the intersection is a single vertex. One of these corresponds to the minimum of $f|_c$ and the other to the maximum of $f|_c$. One of these must be at the vertex mapping to $v$ and this is the minimum. 
\end{proof}

\begin{proofofthm}{prodmorse}
We define a Morse function $f\colon X_1\times X_2\to \RR$ by $f(x_1, x_2) = f_1(x_1)+f_2(x_2)$. This satisfies all the conditions of being a Morse function and is $\phi$-equivariant with respect to the map $\phi(g_1, g_2) = \phi_1(g_1)+\phi_2(g_2)$. 

The link of a vertex $(v_1, v_2)$ in $X_1\times X_2$ is ${\rm Lk}(v_1, X_1)\ast {\rm Lk}(v_2, X_2)$. It is easy to see that the ascending link is the full subcomplex spanned by ${\rm Lk}_{\uparrow}(v_1, X_1)\sqcup {\rm Lk}_{\uparrow}(v_2, X_2)\subset {\rm Lk}(v_1, X_1)\ast {\rm Lk}(v_2, X_2)$, which is equal to ${\rm Lk}_{\uparrow}(v_1, X_1)\ast {\rm Lk}_{\uparrow}(v_2, X_2)$.
\end{proofofthm}

\subsection{Groups of type $F_{n-1}$ but not $F_{n}$}

In \cite{bestvina_problem}, Question 8.5 Brady asks whether there exist groups of type $F_{n-1}$ but not $F_{n}$ which do not contain $\ZZ^2$; he notes that the known examples all contain $\ZZ^{n-2}$. While we are not able to find examples without $\ZZ^2$, we can drastically reduce the rank of a free abelian subgroup as shown by the following. 

\begin{customthm}{B}
For every positive integer $n$, there exists a group of type $F_{n-1}$ but not $F_{n}$ that contains no abelian subgroups of rank greater than $\lceil \frac{n}{3}\rceil$.
\end{customthm}
\begin{proof}
For our general construction, we require 4-tuples $(G_i, X_i, \phi_i, f_i)$ for $i = 1, 2, 3$, where:
\begin{itemize}
\item $G_i$ is a hyperbolic group, 
\item $X_i$ is a CAT(0) cube complex with a free cocompact $G_i$ action, 
\item $\phi_i:G_i\to\ZZ$ is a surjective homomorphism,
\item $f_i$ is a $\phi_i$-equivariant Morse function.
\end{itemize}

We also require that all ascending and descending links of $f_i$ are $(i-2)$-connected but at least one is not $(i-1)$-connected. 

The existence of such a Morse function would show by Theorem \ref{notfp} that $G_i $ has a subgroup of type $F_{i-1}$ not $F_i$, namely $\ker(\phi_i)$. 

Let $G_1$ be a free group of rank 2, $X_1$ a 4-valent tree (the Cayley graph of $G_i$ with respect to generators $a$ and $b$), $\phi_1$ the exponent sum homomorphism with respect to $a$ and $b$, and $f_1$ the map that is linear on edges and whose restriction to the vertices of $X_1$ is $\phi_1$.

Let $(G_2, X_2, \phi_2, f_2)$ be the group, classifying space, homomorphism and Morse function from Proposition \ref{hypf1}.

Let $(G_3, X_3, \phi_3, f_3)$ be the group, classifying space, homomorphism and Morse function from \cite{brady_branched_1999}. We could also use the examples of groups from \cite{kropholler_hyperbolic_2015, lodha_hyperbolic_2017}.

An important point to note is that for $i=1,2,3$, the ascending and descending links for $(G_i, X_i, \phi_i, f_i)$ are isomorphic to $S^{i-1}$. 

Now let $n$ be an integer, let $l = \lfloor\frac{n}{3}\rfloor$, and let $(G_l, X_l, \phi_l, f_l)$ be the 4-tuple with  
$$G_l  = \prod_{i=1}^l G_3, \mbox{ }X_l  = \prod_{i=1}^l X_3, $$
$$\phi_l  = \sum_{i=1}^l\phi_3,\mbox{ } f_l  = \sum_{i=1}^l f_3.$$

By Theorem \ref{prodmorse} the ascending and descending links of $f_l$ are $\ast_{i=1}^l S^2 = S^{3l-1}$ which is $(3l-2)$-connected but not $(3l-1)$-connected. If $n\equiv 0\mod 3$, then $3l=n$, and we define $G = G_l$. If not, let $m$ be the residue of $n\mod 3$ and consider the 4-tuple $(G = G_l\times G_m, X = X_l\times X_m, \phi = \phi_l+\phi_m, f = f_l+f_m)$. The $\phi$-equivariant Morse function $f$ on the cube complex $X$ has ascending and descending links that are all copies of $S^{n-1}$. In all cases, $Ker(\phi)$ is of type $F_{n-1}$ not $F_n$  by Theorem \ref{notfp}. 

The group $G$ contains no free abelian subgroups of rank greater than $\lceil \frac{n}{3}\rceil$ since $G_i$ as a hyperbolic group contains no copy of $\ZZ^2$. 
\end{proof}

\section{Subgroups of type $FP_2$ that cannot be finitely presented}

We will now apply our branched covering technique with carefully chosen flag complexes $\Gamma$ to prove the following.

\begin{customthm}{A}
There exists a non positively curved space $X$ such that $\pi_1(X)$ contains no subgroups isomorphic to $\ZZ^3$, which contains a subgroup of type $FP_2$ not $F_2$. 
\end{customthm}

We will do this using the following steps. 

\begin{enumerate}
\item Start with a connected 2-dimensional tripartite flag complex $L$ such that $\pi_1(L)$ is a perfect group with the the link of every vertex connected and not a point. Take an auxiliary complex $\Gamma  = \mathcal{R}(L)$. Then build the complex $K_{\Gamma}$ in Section \ref{construct}.
\item Define a function $f:K_{\Gamma}\to S^1$ which lifts to a Morse function on universal covers. Examine the ascending and descending links of this Morse function.
\item Take a branched covering of $K_{\Gamma}$ as in Section \ref{almosthyp} to get a complex $X$ with an associated Morse function. Examine the ascending and descending links of this Morse function. This will show that the kernel of $f_*$ is of type $FP_2$. 
\item Prove that the kernel of $f_*$ is not finitely presented. 
\end{enumerate}

\subsection{The complex $\mathcal{R}(L)$}

The construction of this complex is required, the key point of this complex is that it satisfies Proposition \ref{simpcon1}. This is required to make sure that the fundamental group of the links is not changed in the branching process. 

Firstly, we prove that none of the assumptions from the first step are restrictive. We can realise any finitely presented group as the fundamental group of a finite connected 2-dimensional simplicial complex. It is also well known that the barycentric subdivision of a 2-dimensional simplicial complex is flag and we can put an obvious tripartite structure, on the barycentric subdivision, labelling vertices by the dimension of the corresponding cell. So we can realise any group as the fundamental group of a connected 2-dimensional tripartite flag complex. 

Given a connected 2-dimensional tripartite flag complex there is a homotopy equivalent complex of the same form such that the link of every vertex is connected and not a point. 

To see this, first note that if there is a vertex where the link is just a single point, then we can contract this edge without changing the homotopy type of the complex. Next, note that if there is a vertex $x$ with disconnected link, then we perform the following procedure: pick two vertices $v$, $w$ in two different components of the link, add an extra vertex $y$ and connect it to $v,w$ and $x$ while also adding two triangles $[v,y,x]$ and $[u,y,x]$. The result is that we have reduced the number of components of the link at $x$, without adding any extra components to the link of $v$ or $w$, and the link of $y$ is connected. We have also not changed the homotopy type, since we have added a contractible space glued along a contractible subspace. Repeating this procedure, we can make sure that the link of every vertex is connected. 

\begin{definition}
For a simplicial complex $L$ the {\em octahedralisation} $S(L)$ is defined as follows. For each vertex $v$ of $L$, let $S^0_v = \{v^+, v^-\}$ be a copy of $S^0$. For every simplex $\sigma$ of $L$ take $S_{\sigma} = \ast_{v\in \sigma}S^0_v$. If $\tau<\sigma$, then there is a natural map $S_{\tau}\to S_{\sigma}$. $S(L) = \cup_{\sigma< L} S_{\sigma}/\sim$, where the equivalence relations $\sim$ is generated by the inclusions $S_{\tau}\to S_{\sigma}.$
\end{definition}

The complex $S(L)$ can also be seen as the link of the vertex in the Salvetti complex for the RAAG defined by $L$. If $L$ has an $n$-partite structure then there is a natural $n$-partite structure on $S(L)$. Let $L$ be contained in $V_1\ast\dots\ast V_n$. Then $S(L)$ is contained in $S(V_1)\ast\dots\ast S(V_n)$. 

\begin{rem}
The map defined by $S^0_v\to \{v\}$ extends to a retraction of $S(L)$ to $L$ (not a deformation retraction). 
\end{rem}

In particular, if $\pi_1(L)\neq 0$, then $\pi_1(S(L))\neq 0$. 

It is proved in \cite{bestvina_morse_1997} that if $L$ is a flag complex, then $S(L)$ is a flag complex. 

\begin{lem}\label{trivh1}
Assume $L$ is a connected simplicial complex and $\lk(v, L)$ is connected and not equal to a point for all vertices $v\in L$. If $H_1(L) = 0$, then $H_1(S(L)) = 0$. 
\end{lem}
\begin{proof}
Let $L^+$ be the full subcomplex of $S(L)$ spanned by the set $\{v^+:v\in L\}$, and let $L^-$ be defined similarly. Let $N^+$ be the interior of the star of $L^+$ and $N^-$ the interior of the star of $L^-$. It is clear that $S(L)$ is contained in $N^-\cup N^+$. We can also see that $N^+\cap N^-$ is the union of the open simplices which are contained in neither $L^+$ nor $L^-$. Here we are using the fact that $N^+$ is homotopy equivalent to $L$ and similarly $N^-$ is homotopy equivalent to $L$. 

By considering the Mayer-Vietoris sequence for $N^-$ and $N^+$ we see that $H_1(S(L))$ is isomorphic to $\tilde{H}_0(N^+\cap N^-)$, so we must prove that $N^+\cap N^-$ is connected. 

Let $x, y$ be points of $N^+\cap N^-$. We can always connect $x$ and $y$ to open edges contained in $N^+\cap N^-$; label these edges $e_x$ and $e_y$. Let $v_x$ be the end point of $e_x$ in $L^+$ and $w_x$ the end point in $L^-$; define vertices $v_y$ and $w_y$ similarly. Let $\mathfrak{v} = (v_x = v_0, v_1, \dots, v_n = v_y)$ be a sequence of vertices in $L^+$ corresponding to a geodesic in $L^{(1)}$ from $v_x$ to $v_y$. 

The vertices $w_x$ and $w_y$ have corresponding vertices in $L^+$, these are adjacent to $v_x$ and $v_y$ respectively. We split into 4 cases by whether these vertices are on the geodesic $\mathfrak{v}$. We will define a path $P$ in each of these four cases. 
\begin{enumerate}
\item If $w_x = v_1$ and $w_y = v_{n-1}$, then let $P = v_0, v_1, \dots, v_n$.
\item If $w_x = v_1$ and $w_y \neq v_{n-1}$, then let $P = v_0, v_1, \dots, v_n, w_y$.
\item If $w_x \neq v_1$ and $w_y = v_{n-1}$, then let $P = w_x, v_0, v_1, \dots, v_n$.
\item If $w_x \neq v_1$ and $w_y \neq v_{n-1}$, then let $P = w_x, v_0, v_1, \dots, v_n, w_x$.
\end{enumerate}

We will now describe how to get a sequence of edges $e_x = e_0, \dots, e_m = e_y$ such that $e_i, e_{i+1}$ are on a 2-simplex in $L$. The corresponding sequence of edges in $N^+\cap N^-$ will give a path from $e_x$ to $e_y$. Thus completing the proof of the lemma. 

A path in the link of a vertex $v$ can be viewed as a sequence of edges in $L$ such that adjacent edges are on a 2-simplex.  

We will relabel the path $P$ to $a_1, b_1, \dots$. Let $A_i$ be a path in the link of $b_i$ from $a_i$ to $a_{i+1}$. Let $B_i$ be a path in the link of $a_i$ from $b_i$ to $b_{i+1}$. This can be done since the link of every vertex is connected. The sequence of edges $A_1, B_2, A_2, \dots$ defines the sequence of edges we require. The key idea is encapsulated in Figure \ref{connlinkgame}, where the curved arcs correspond to the paths $A_i$, $B_i$. 
\end{proof}

\begin{figure}
\center
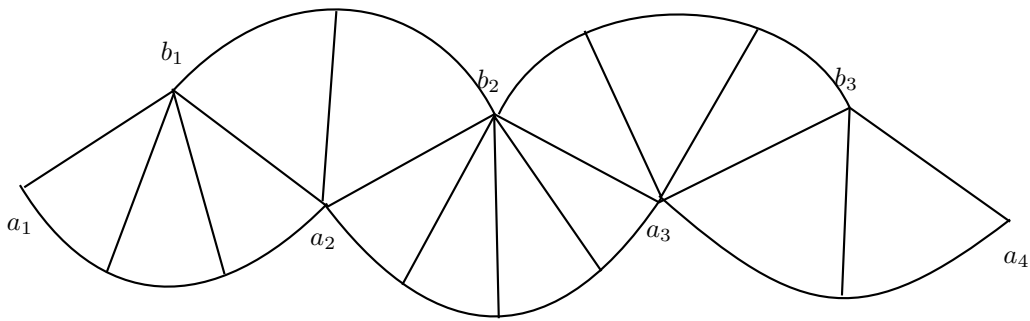
\caption{The key idea from Lemma \ref{trivh1}.}
\label{connlinkgame}
\end{figure}

We can now define the complex $\mathcal{R}(L)$. 

\begin{enumerate}
\item Let $L$ be a 2-dimensional tripartite flag complex with $\pi_1(L)$ perfect, such that the link of every vertex in $L$ is connected and not a point. Label the sets of vertices from the tripartite structure $V_0, V_1, V_2$.
\item Construct $S(L)$ as above. This is a 2-dimensional tripartite flag complex with $\pi_1(L)$ perfect. 
\item Add 3 extra vertices $v^0, v^1, v^2$ which are of type $0, 1,2$ respectively. Connect $v^i$ to all the vertices not of type $i$. Define $\mathcal{R}(L)$ to be the flag completion of the resulting complex. 
\end{enumerate}

Take a simplicial complex $L$ as above and let $\Gamma = \mathcal{R}(L)$. We can now construct the cubical complex $K_{\Gamma}$ from Section \ref{construct}.

\subsection{The Morse function $f$}

As noted in Remark \ref{keyrem}, we can view $K_{\Gamma}$ as a subcomplex of $\Theta_1\times\Theta_2\times\Theta_3$, where $\Theta_i$ is a graph that has two vertices $0, 1$ and has edges labelled by the vertices of $\Gamma$ as well as one extra edge labelled $v_0^i$, such that each edge runs from $0$ to $1$. 

We define a Morse function on the product by putting an orientation on each edge of $\Theta_i$ as follows: if it is an edge corresponding to a vertex of $S(L)\subset\Gamma$ we orient it towards $1$, while for the vertices $v^i$ and $v_0^i$ we orient towards $0$. Now put an orientation on $S^1$ and map each graph by its orientation; we then extend linearly across cubes. Restricting this map to $K_{\Gamma}$ we get a map $g$, and by lifting to universal covers we get a Morse function $f:\tilde{K}_{\Gamma}\to\RR$ which is $g_*$-equivariant. 

Let the vertices of type $i$ in $S(L)\subset\Gamma$ be the set $V_i^+$ and let $V_i^- = \{v^i, v_0^i\} =: S^0_i$. 

The ascending and descending links of this Morse function are given in Table \ref{changeme}. 

\begin{notation}
Given a simplicial complex $\Lambda$ and a subset $S$ of the vertices of $\Lambda$, $\Lambda(S)$ denotes the full subcomplex of $\Lambda$ spanned by the vertices in $S$. 
\end{notation}

\begin{table}
\begin{center}\begin{tabular}{c|c|c}Vertex & $\lk_{\uparrow}$ & $\lk_{\downarrow}$ \\
\hline
$(0,0,0)$ & $\Gamma(V_1^+\cup V_2^+\cup V_3^+) = S(L)$ & $S^2 = S^0_1\ast S^0_2\ast S^0_3$ \\
$(0,0,1)$ & $\Gamma(V_1^+\cup V_2^+)\ast S^0_3$ & $V_3^+\ast S^0_1\ast S^0_2 $ \\
$(0,1,0)$ & $\Gamma(V_1^+\cup V_3^+)\ast S^0_2$ & $V_2^+\ast S^0_1\ast S^0_3 $ \\
$(0,1,1)$ & $V_1^+\ast S^0_2\ast S^0_3 $ & $\Gamma(V_2^+\cup V_3^+)\ast S^0_1$ \\
$(1,0,0)$ & $\Gamma(V_2^+\cup V_3^+)\ast S^0_1$ & $V_1^+\ast S^0_2\ast S^0_3 $ \\
$(1,0,1)$ & $V_2^+\ast S^0_1\ast S^0_3 $ & $\Gamma(V_1^+\cup V_3^+)\ast S^0_2$ \\
$(1,1,0)$ & $V_3^+\ast S^0_1\ast S^0_2 $ & $\Gamma(V_1^+\cup V_2^+)\ast S^0_3$ \\
$(1,1,1)$ & $S^2 = S^0_1\ast S^0_2\ast S^0_3$ & $\Gamma(V_1^+\cup V_2^+\cup V_3^+) = S(L)$
\end{tabular} 
\end{center}
\caption{The ascending and descending links of the Morse function $f:K_{\Gamma}\to S^1$.}
\label{changeme}
\end{table}

\begin{prop}\label{simpcon1}
Given a complex of the form $V_i^+\ast V_j^-\ast S^0_k$ or $\Gamma(V_i^+\cup V_j^+)\ast S^0_k$ there is an ordering $v_1, v_2, \dots, v_q$ on the set $V_j^s$ such that $St(v_m)\cap(\cup_{l<m}St(v_l))$ is connected and $\pi_1(St(v_m)\cap(\cup_{l<m}St(v_l)))\neq0$ for $1<m\leq q$.
\end{prop}
\begin{proof}
In the case $V_j^s = V_j^- = S^0_j$ we have 2 vertices $v_1$ and $v_2$ and $St(v_1)\cap St(v_2) = S^0_k\ast V_i^+$, which is connected but not simply connected. So either ordering will do. 

In the case $V_j^s = V_j^+$, let $V_j = \{w_1, \dots, w_n\}$ and $V_j^+ = \{w_1^+, w_1^-, w_2^+, \dots, w_n^-\}$. 

The subgraph $Z = L(V_i\cup V_j)$ is connected since $L$ is connected and the link of every vertex is connected and not a point. We can thus assume that $St(w_l, Z)\cap (\cup_{m<l}St(w_m))\neq \emptyset$.

Noting that $St(v^s, S(L)) = C(S(\lk(v, L)))$. We can see that $St(v^s, S(L))\cap St(w^t, S(L)) = S(St(v, L)\cap St(w, L))$. Thus ordering $V_j^+ = \{w_1^+, w_1^-, w_2^+, \dots, w_n^-\}$ we can see that $V_i^+\cap St(v_l, S(L))\cap(\cup_{m<l}St(v_m, S(L)))$ contains at least two points. 

Noting that 
$$St(v_l, S(L))\cap(\bigcup_{m<l}St(v_m, S(L))) = 
\Big(V_i^+\cap \Big(St(v_l, S(L))\Big)\cap\Big(\bigcup_{m<l}St(v_m, S(L))\Big)\Big)\ast S^0_k$$ 
we can see that this is connected but not simply connected. 
\end{proof}

\begin{rem}\label{simpcon}
In the above proof we are gluing contractible complexes along connected complexes. This shows that all the complexes in the statement of Proposition \ref{simpcon1} are simply connected.
\end{rem}

\begin{rem}\label{4cyc}
At each stage in the above proof $St(v_l, S(L))\cap(\cup_{m<l}St(v_m, S(L)))$ could be covered by cycles of length 4, since it is the join of a discrete set and a copy of $S^0$. 
\end{rem}

\subsection{Almost hyperbolisation and the Morse function $h:X\to S^1$}

We use the almost hyperbolisation technique from Section \ref{almosthyp} to get a branched cover $X$ of $K_{\Gamma}$, recall that there is a natural length preserving map $b\colon X\to K_{\Gamma}$. We define a function $h = g\circ b\colon X\to S^1$, lifting to universal covers we get a Morse function. In what follows $G$ is the fundamental group of $X$ which does not contain any copies of $\ZZ^3$. In what follows, $H := \ker(h_*:G\to\ZZ)$.

It is worth noting that in the almost hyperbolisation procedure we ensured that loops of length 4 in the link of a vertex have connected preimage.

\subsubsection{Ascending and descending links of $h$}

Recall that $K_{\Gamma}\subset \Theta_1\times\Theta_2\Theta_3$. We distinguish between 4 types of vertices in $X$ and label them as follows:

Vertices of type $A$ are those which map to $(0,0,0)$ or $(1,1,1)$.

Vertices of type $B$ are those which map to $(0,0,1)$ or $(1,0,1)$.

Vertices of type $C$ are those which map to $(1,0,0)$ or $(1,1,0)$.

Vertices of type $D$ are those which map to $(0,1,0)$ or $(0,1,1)$.

For a vertex in $X$ the ascending (descending) link is the preimage of the ascending (descending) link of the corresponding vertex in $K_{\Gamma}$.

Type $A$ vertices are disjoint from the branching locus, so a small neighbourhood of each lifts to $X$ and the ascending and descending links are isomorphic to those of the corresponding vertex of $K_{\Gamma}$. We claim that for vertices {\em not} of type $A$ the ascending and descending links are simply connected. We will prove this in the case of a vertex of type $B$ and the ascending link; the other cases are similar. 

A vertex $x$ of type $B$ is on a lift of $\Theta_1$. We may assume that $x$ maps to $(0,0,1)$. Now, $\lk_{\uparrow}(0,0,1) = \Gamma(V_1^+\cup V_2^+)\ast S^0_3$. Let us consider the preimage of $\lk_{\uparrow}(0,0,1)$ in $\lk(x, X)$. To envisage this, we start by removing the vertices in $V_1$ and taking the covering of the remaining space coming from the derivative map $b_*$ of $b$, then add back the vertices of $V_1$. By Remark \ref{simpcon} we can cover $\Gamma(V_1^+\cup V_2^+)\ast S^0_3$ by $St(v)$  as $v$ runs over vertices in $V_1^+$. We noted in Remark \ref{4cyc} that we can construct this cover in such a way that each $St(v_m)\cap(\cup_{l<m}St(v_l))$ is connected and covered by loops of length 4 with non-empty intersection. Specifically, if $S^1_4$ and $S^1_{4'}$ are 2 loops of length 4 in $St(v_m)\cap(\cup_{l<m}St(v_l))$, then $S^1_4\cap S^1_{4'}\supset S^0_3$. 

In the procedure of passing to the branched covering $b:X\to K_{\Gamma}$. Associated to each vertex of $X$ there is a derivative map $b_{\lk(v)}:\lk(v, X)\to \lk(b(v), K_{\Gamma})$. Each of the 4 cycles above has connected preimage under the map $b_{\lk(v)}$, therefore the preimage of $St(v_m)\cap(\cup_{l<m}St(v_l))$ is connected. Upon taking the completion, we replace the vertices in $V_1^+$; this corresponds to coning off the lifts of their links. Thus we see $\lk_{\uparrow}(x)$ is made from a sequence of contractible spaces glued along connected subspaces, and so is simply connected.

\subsection{Proof that the kernel is not finitely presented}

To prove that $H$ is not finitely presented we need the following lemma.

\begin{lem}\label{nontriv}
$\pi_1(X_I)\neq 0$ for all compact intervals $I\subset \RR$
\end{lem}
\begin{proof}
For the purposes of this proof let $Y = \widetilde{K_{\Gamma}}$. In our case $\Gamma^{\dagger} = S^2$ and $\Gamma^{\ast} = \mathcal{R}(L)$. By Theorem \ref{buxthm} we know that the kernel of $g_*$ is not finitely presented. Since the kernel of $g_*$ acts cocompactly on the level set $Y_I$ we can see that $Y_I$ is not simply connected. 

Let $\gamma:S^1\to Y_I$ be a non-trivial loop. Then there is a larger interval $J'$ such that $\gamma$ is trivial in $Y_{J'}$. Let $J' = [a,b]$. We can assume that $\gamma$ is non-trivial in $J = [a+\epsilon, b]$ or $J = [a, b-\epsilon]$ for all $\epsilon>0$; assume the latter. There is a sequence of vertices $v_1, v_2, \dots$ such that $Y_{J'} = Y_{J}\cup C(\lk_{\uparrow}(v_i))$. There is an integer $m$ such that $\gamma$ is trivial in $Y_{J}\cup_{i\leq m} C(\lk_{\uparrow}(v_i))$ but not trivial in $Y_{J}\cup_{i< m} C(\lk_{\uparrow}(v_i)) = \widehat{Y_{J}}$. 

Since adding $C(\lk_{\uparrow}(x_m))$ changes the fundamental group ($\gamma$ becomes trivial), we can find a non-trivial loop $\hat{\gamma}$ in  $\widehat{Y_{J}}$ which is also contained in $\lk_{\uparrow}(x_m)$. We also know that it is contained in ${Y_{J}}$ since $\lk_{\uparrow}(v_i)\cap \lk_{\uparrow}(v_j) = \emptyset$ whenever $i\neq j$, as the restriction of an affine map to a cube can only have one maximum and one minimum or it is constant on a subcube. 

Since adding $C(\lk_{\uparrow}(x_m))$ changes $\pi_1$, we see that $\pi_1(\lk_{\uparrow}(x_m))\neq 0$, so $x_m$ is a vertex mapping to $(0,0,0)$ and $\hat{\gamma}$ bounds a disc in $Y$ which does not intersect $\tilde{L}$. It follows that this disc lifts to $\tilde{X}$ under the branched covering $\tilde{b}$ coming from the following commutative diagram:

\begin{figure}[h!]
\center
\begin{tikzcd}
  \tilde{X} \arrow[r, "\tilde{b}"] \arrow[d]
    & Y \arrow[d, ] \\
  X \arrow[r, "b" ]
&K_{\Gamma} 
\end{tikzcd}
\end{figure}

The boundary of this lifted disc is in $X_J$; call this loop $\mu$. If $\mu$ bounded a disc in $X_J$, then we could map this to $Y_{J}$ via $\tilde{b}$, but this would imply that $\tilde{b}(\mu) = \hat{\gamma}$ bounds a disc in $Y_{J}$, which it does not. Thus $\mu$ is non-trivial in $\pi_1(X_J)\neq 1$. Since $\pi_1(X_I)\rightarrow\pi_1(X_J)$ is surjective we deduce by Theorem \ref{bbmorse} that $\pi_1(X_I)\neq 1$. 
\end{proof}

\begin{thm}
$H = \ker(g_*)$ is not finitely presentable.
\end{thm}
\begin{proof}
Assume that $H$ is finitely presented. 

$H$ acts cocompactly on $X_I$ so we can add finitely many 2-cells to the quotient to gain fundamental group $H$. Taking a universal cover of the space obtained in this way we arrive at $X_I$ with finitely many $H$-orbits of $2$-cells attached, which is simply connected. In other words, there are finitely many $H$-orbits of loops which generate $\pi_1(X_I)$. 

The direct limit of  all the $X_I$ is the space $X$ which is CAT(0) and in particular contractible. We can pass to a larger interval such that the $H$-orbits of loops which generate $\pi_1(X_I)$ are trivial. In other words, the map $\pi_1(X_I)\to\pi_1(X_J)$ is trivial. But it is also surjective by Theorem \ref{bbmorse} and because we have assumed that all the ascending and descending links are connected. Thus $\pi_1(X_J)$ is trivial, however we know this not to be the case by Lemma \ref{nontriv}. 
\end{proof}

\bibliographystyle{plain}
\bibliography{bibshort,otherref}

\end{document}